\documentclass[11pt]{article}
\usepackage{amscd}
\usepackage{amsfonts}
\usepackage{amsmath}
\usepackage{amssymb}
\usepackage{amsthm}
\usepackage{bbm}
\usepackage{CJK}
\usepackage{fancyhdr}
\usepackage{graphicx}
\usepackage{hyperref}
\usepackage{indentfirst}
\usepackage{latexsym}
\usepackage{mathrsfs}
\usepackage{xypic}

\newtheorem{theorem}{Theorem}[section]
\newtheorem{lemma}[theorem]{Lemma}
\newtheorem{definition}[theorem]{Definition}
\newtheorem{proposition}[theorem]{Proposition}

\newtheorem{cor}[theorem]{Corollary}
\newtheorem{remark}[theorem]{Remark}

\usepackage[top=1in,bottom=1in,left=1.25in,right=1.25in]{geometry}
\def\<{\langle}
\def\>{\rangle}

\def\c{\cdot}

\def\d{\delta}

\def\o{\otimes}

\date{}
\begin{document}
\renewcommand{\baselinestretch}{1.2}
\renewcommand{\arraystretch}{1.0}
\title{\bf On 3-Lie algebras with  a derivation }
\date{}
\author{{\bf Shuangjian Guo$^{1}$, Ripan Saha$^{2}$\footnote
        { Corresponding author:~~ripanjumaths@gmail.com} }\\
{\small 1. School of Mathematics and Statistics, Guizhou University of Finance and Economics} \\
{\small  Guiyang  550025, P. R. of China} \\
{\small 2. Department of Mathematics, Raiganj University } \\
{\small  Raiganj, 733134, West Bengal, India}}
 \maketitle
\begin{center}
\begin{minipage}{13.cm}

{\bf \begin{center} ABSTRACT \end{center}}
In this paper, we study 3-Lie algebras with derivations. We call the pair consisting of a 3-Lie  algebra and a distinguished derivation by the 3-LieDer pair. We define a cohomology theory for 3-LieDer pair with coefficients in a representation. We study central extensions of a 3-LieDer pair and show that central extensions are classified by the second cohomology of the 3-LieDer pair with coefficients in the trivial representation. We generalize Gerstenhaber's  formal deformation theory
to 3-LieDer pairs in which we deform both the 3-Lie  bracket and the distinguished derivation.
 \smallskip

 \smallskip

 \end{minipage}
 \end{center}
 \normalsize\vskip0.5cm

\section{Introduction}
\def\theequation{\arabic{section}. \arabic{equation}}
\setcounter{equation} {0}
3-Lie algebras are special types of $n$-Lie algebras and have close relationships with many important fields in mathematics
and mathematical physics \cite{BL08, BL8}. The structure of 3-Lie algebras is closely linked to the supersymmetry and gauge
symmetry transformations of the world-volume theory of multiple coincident $M2$-branes and is applied to the study of
the Bagger-Lambert theory. Moreover, the $n$-Jacobi identity can be regarded as a generalized Plucker relation in the physics
literature.  In particular, the metric 3-Lie algebras, or more generally, the 3-Lie algebras with invariant symmetric bilinear forms attract even more attention in physics. Recently, many more properties and structures of 3-Lie algebras have been developed, see  \cite{BG16, BW12,  DB18, L16, ST18, X19, Z15}  and references cited therein.

Derivations of types of algebra provide many important aspects of the algebraic structure. For example, Coll, Gertstenhaber, and Giaquinto \cite{C89} described explicitly a deformation formula for algebras whose Lie algebra of derivations contains the unique non-abelian Lie algebra of dimension two. Amitsur \cite{AS57, AS82} studied derivations of central simple algebras. Derivations are also used to construct homotopy Lie algebras \cite{V05} and play an important role in the study of differential Galois theory \cite{M94}. One may also look at some interesting roles played by derivations in control theory and gauge theory in quantum field theory \cite{A12}. In \cite{DL16}, the authors studied algebras with derivations from an operadic point of view. Recently, Lie algebras with derivations (called LieDer pairs) are studied from a cohomological point of view \cite{T19} and extensions, deformations of LieDer pairs are considered. The results of \cite{T19} have been extended to associative algebras and Leibniz algebras with derivations in \cite{D20} and \cite{DL20}.

The deformation is a tool to study a mathematical object by deforming it into a family of the same kind of objects depending on a certain parameter. The deformation theory was introduced by Gerstenhaber for rings and algebras \cite{G63, G64}, and by Zhang for 3-Lie color algebras \cite{Z15}. They studied 1-parameter formal deformations and established the connection between the cohomology groups and infinitesimal deformations. Motivated by Tang's \cite{T19} terminology of LieDer pairs.  Due to the importance of  3-Lie algebras, cohomology, and deformation theories, Our main objective of this paper is to study the cohomology and deformation theory of 3-Lie algebra with a derivation.

The paper is organized as follows. In Section 2, we define a cohomology
theory for 3-LieDer pair with coefficients in a representation. In Section 3,  we study central extensions of a 3-LieDer pair and show that isomorphic classes of central extensions are classified by the
second cohomology of the 3-LieDer pair with coefficients in the trivial representation. In Section 4,  we study formal one-parameter deformations of 3-LieDer pairs in which we deform
both the 3-Lie  bracket and the distinguished derivations.

Throughout this paper, we work over the field $\mathbb{F}$ of characteristics $0$.

  \section{  Cohomology of 3-LieDer pairs}
\def\theequation{\arabic{section}. \arabic{equation}}
\setcounter{equation} {0}
 In this section, we define a cohomology theory for 3-LieDer pair with coefficients in a representation.

\begin{definition} (\cite{F85})
A 3-Lie algebra is a tuple $(L, [\c, \c, \c])$ consisting of a vector space $L$, a 3-ary skew-symmetric
operation $[\c,\c,\c]: \wedge^{3}L\rightarrow L$  satisfying the following Jacobi identity
\begin{eqnarray}\label{3-lie identity}
[x, y, [u,v,w]]&=&[[x,y,u], v, w] +[u, [x,y,v], w]+[u, v,[x,y,w]],
\end{eqnarray}
for any $x, y, u, v,w \in L$.
\end{definition}
\begin{definition} (\cite{K87})
A representation of a  3-Lie algebra $(L, [\c, \c, \c])$ on the vector space $M$  is a linear map $\rho: L \wedge L \rightarrow \mathfrak{gl}(M)$, such that for any $x, y, z, u\in L$, the following equalities are satisfied
\begin{eqnarray*}
&& \rho([x,y,z], u)=\rho(y, z)\rho(x,u)+\rho(z, x)\rho(y,u)+\rho(x, y)\rho(z,u),\\
&& \rho(x, y)\rho(z,u)=\rho(z, u)\rho(x,y)+\rho([x,y,z], u)+\rho(z, [x,y,u]).
\end{eqnarray*}
Then $(M, \rho)$ is called a representation of $L$, or $M$ is an $L$-module.
\end{definition}

\begin{definition}(\cite{F85})
 Let $(L, [\c, \c, \c])$ be a  3-Lie algebra. A derivation on $L$ is
given by a linear map $\phi_L: L \rightarrow L$ satisfying
\begin{eqnarray*}
\phi_L([x, y, z])=[\phi_L(x), y, z]+[x, \phi_L(y), z]+[x, y,  \phi_L(z)],~~~~~\forall x, y, z \in L.
\end{eqnarray*}
\end{definition}
We  call the pair $(L, \phi_L)$ of a 3-Lie algebra and a derivation by a 3-LieDer pair.

\begin{remark}
Let $(L, [\c, \c, \c])$ be a $3$-Lie algebra. For all $x_1, x_2 \in L$, the map defined by 
$$ad_{x_1, x_2} x := [x_1, x_2, x],~\text{for all}~x\in L,$$
is called the adjoint map. From the Equation \ref{3-lie identity}, it is clear that $ad_{x_1, x_2}$ is a derivation. The linear map $ad : L\wedge L \to \mathfrak{gl}(L)$ defines a representation  of $(L, [\c, \c, \c])$ on itself. This representation is called the adjoint representation.
\end{remark}

\begin{definition}
 Let $(L, \phi_L)$ be a 3-LieDer pair.  A representation of $(L, \phi_L)$ is given by $(M, \phi_M)$ in which $M$ is a  representation of $L$ and $\phi_M: M\rightarrow M$ is a linear map satisfying
 \begin{eqnarray*}
&& \phi_M(\rho(x,y)(m))=\rho(\phi_L(x),y)(m)+\rho(x,\phi_L(y))(m)+\rho(x,y)(\phi_M(m)),
 \end{eqnarray*}
for all $x, y\in L$ and $m\in M$.
\end{definition}
\begin{proposition}
 Let $(L, \phi_L)$ be a 3-LieDer pair  and  $(M, \phi_M)$ be a representation of it. Then $(L\oplus M,  \phi_L\oplus\phi_{M})$ is a 3-LieDer pair where the 3-Lie algebra bracket on $L\oplus M$  is given by the semi-direct product
\begin{eqnarray*}
[(x, m), (y, n), (z, p)]=([x, y, z], \rho(y,z)(m)+\rho(z,x)(n) + \rho(x,y)(p)),
\end{eqnarray*}
for any $x, y, z\in L$ and $m, n, p\in M$.
\end{proposition}
{\bf Proof.}    It is known that $L\oplus M$ equipped with the above product is a 3-Lie algebra.
Moreover,  we have
\begin{eqnarray*}
&&(\phi_L\oplus\phi_{M})([(x, m), (y, n), (z, p)])\\
&=& (\phi_L([x, y, z]), \phi_{M}(\rho(y,z)(m))+\phi_{M}(\rho(z,x)(n))+\phi_M(\rho(x,y)(p)))\\
&=&  ([\phi_L(x), y, z], \rho(y,z)(\phi_M(m))+\rho(\phi_L(x),z)(n) + \rho(\phi_L(x),y)(p))\\
&&+ ([x, \phi_L(y), z], \rho(\phi_L(y),z)(m)+\rho(z,x)(\phi_M(n)) + \rho(x,\phi_L(y))(p))\\
&&+([x, y, \phi_T(z)], \rho(y,\phi_L(z))(m)+\rho(z,\phi_L(x))(n) + \rho(x,y)(\phi_{M}(p)))\\
&=&  [(\phi_L\oplus\phi_{M})(x, m), (y, n), (z, p)]+[(x, m), (\phi_L\oplus\phi_{M})(y, n), (z, p)]\\
&&+[(x, m), (y, n), (\phi_L\oplus\phi_{M})(z, p)].
\end{eqnarray*}
Hence the proof is finished.   \hfill $\square$

Recall from \cite{T95} that let $\rho$ be a representation of $(L,[\c, \c,\c])$ on $M$.  Denote by $C^n(L,M)$ the set of all $n$-cochains and defined as
$$C^n(L,M) = \text{Hom}((\wedge^2 L)^{\otimes n-1}, M),~n\geq 1.$$

Let $d^n: C^{n}(L, M) \rightarrow C^{n+1}(L, M)$ be defined by
\begin{align*}
     &d^n f(X_1,\ldots, X_n, x_{n+1})\\
    =&(-1)^{n+1}\rho(y_{n}, x_{n+1})f(X_1,\ldots, X_{n-1}, x_{n})\\
     &+(-1)^{n+1}\rho( x_{n+1}, x_{n})f(X_1,\ldots,X_{n-1},y_{n})\\
     &+\sum_{j=1}^n(-1)^{j+1}\rho(x_{j}, y_{j}) f(X_1,\ldots,\hat{X_{j}},\ldots, X_{n}, x_{n+1})\\
     &+\sum_{j=1}^n(-1)^{j} f(X_1,\ldots,\hat{X_{j}},\ldots, X_n, [x_{j}, y_{j}, x_{n+1}]),\\
     &+ \sum_{1\leq j<k\leq n}(-1)^{j}f(X_1,\ldots,\hat{X_{j}},\ldots, X_{k-1}, [x_{j}, y_{j}, x_{k}]\wedge y_k\\
    & +x_k\wedge [x_{j}, y_{j}, x_{k}],X_{k+1},\ldots,X_{n}, x_{n+1}),
\end{align*}
for all $X_i=x_i\wedge y_i\in \otimes^2L,i=1, 2, \ldots, n$ and $x_{n+1}\in L$, it was proved that  $d^{n+1}\circ d^n=0$. Therefore, $(C^\ast (L, M), d^\ast)$ is a cochain complex.

Observe that for trivial representation coboundary maps $d^1$ and $d^2$ are explicitly given as follows:
$$d^1(f)(a,b,c)= [f(a),b,c] + [a, f(b), c] + [a, b, f(c)]- f([a,b,c]),~f\in C^1(L, M).$$
$$d^2(f)(a,b,c,d,e) = [a,b, f(c,d,e)]- f([a,b,c],d,e) + f(a,b,[c,d,e])-[f(a,b,c),d,e],~f\in C^2(L, M).$$

  In \cite{R05}, the graded space $C^{\ast} (L, L) = \bigoplus_{n\geq 0}C^{n+1} (L, L)$ of cochain groups carries a degree -1 graded Lie bracket given by
$[f, g] = f \circ g- (-1)^{mn} g \circ  f$, for $f \in C^{m+1} (L, L), g \in C^{n+1} (L, L)$, where $f \circ g \in C^{m+n+1} (L, L)$, and defined as follows:
\begin{eqnarray*}
&&f\circ g(X_1,\ldots, X_{m+n}, x)\\
&=& \sum_{k=1}^{m}(-1)^{(k-1)n} \sum_{\sigma\in \mathbb{S}(k-1, n)}f(X_{\sigma(1)}, \ldots, X_{\sigma(k-1)}, g(X_{\sigma(k)}, \cdots, X_{\sigma(k+n-1)}, x_{k+n})\\
&&\wedge y_{k+n},X_{\sigma(k+n+1)}, \ldots, X_{\sigma(m+n)}, x) \\
&&+\sum_{k=1}^{m}(-1)^{(k-1)n} \sum_{\sigma\in \mathbb{S}(k-1, n)}(-1)^{\sigma} f(X_{\sigma(1)}, \ldots, X_{\sigma(k-1)},x_{k+n}\\
&&\wedge g(X_{\sigma(k)}, \ldots, X_{\sigma(k+n-1)},y_{k+n}), X_{k+n+1}, \ldots, X_{m+n}, x)\\
&&\sum_{\sigma \in \mathbb{S}(m, n)}(-1)^{mn}(-1)^{\sigma}f(X_{\sigma(1)}, \ldots, X_{\sigma(m)}, g(X_{\sigma(m+1)}, \ldots, X_{\sigma(m+n-1)},X_{\sigma(m+n)}, x)),
\end{eqnarray*}
for all $X_i=x_i\wedge y_i\in \otimes^2L,i=1, 2, \ldots, m+n$ and $x\in L$. Here $\mathbb{S}(k-1, n)$ denotes the set of all $(k-1, n)$-shuffles. Moreover, $\mu : \otimes^3L\rightarrow L$ is a 3-Lie bracket if and only if $[\mu, \mu] = 0$, i.e. $\mu$ is a
Maurer-Cartan element of the graded Lie algebra $(C^{\ast} (L, L), [\c, \c]$.   where $\mu$ is considered as an element in $C^2 (L, L)$. With this notation, the differential (with coefficients
in $L$) is given by
\begin{eqnarray*}
df=(-1)^{n}[\mu, f], ~~~~\text{for all}~ f\in C^{n} (L, L).
\end{eqnarray*}

In the next, we introduce cohomology for a 3-LieDer pair with coefficients in a representation.

 Let $(L,  \phi_L)$ be a 3-LieDer pair and  $(M, \phi_{M})$ be a representation of it.
  For any $n\geq 2$,  we define cochain groups for 3-LieDer pair as follows:
  $$C^n_{\text{3-LieDer}}(L, M) := C^n(L, M)\oplus C^{n-1}(L, M).$$
  Define the space
$C^0_{\text{3-LieDer}} (L,M)$ of $0$-cochains to be $0$ and the space $C^1_{\text{3-LieDer}} (L,M)$ of 1-cochains to be Hom$(L,M)$. 
  Note that $\mu = [\c,\c,\c] \in C^2(L, L)$ and derivation $\phi_L \in C^1(L, L) $. Thus, the pair $(\mu, \phi_L) \in C^2_{\text{3-LieDer}}(L, L)$.
 To define the coboundary map for $3$-LieDer pair, we need following map $\delta : C^n(L, M)\rightarrow C^n(L, M)$ by
 \begin{eqnarray*}
\delta f=\sum_{i=1}^n f\circ (Id_L\o \c \c \c\o \phi_L\o\c \c \c \o Id_L)-\phi_{M}\circ f.
 \end{eqnarray*}

The following lemma shows maps $\partial$ and $\delta$ commute, and is  useful to define the coboundary operator of the cohomology of
3-LieDer pair.
\begin{lemma}
The map $\delta$ commute with $d$, i.e, $d\circ \delta=\delta\circ d$.
\end{lemma}
\begin{proof}
Note that in case of self representation, that is, when $(M, \phi_M) = (L, \phi_L)$, we have 
$$\delta(f) = - [\phi_L, f],~\text{for all}~f\in C^n(L, L).$$
Therefore, we have
\begin{align*}
(d\circ \delta)(f) &= -d[\phi_L, f]\\
                            &= (-1)^n[\mu, [\phi_L, f]]\\
                            & = (-1)^n [[\mu,\phi_L], f] + (-1)^n[\phi_L,[\mu,f]]\\
                            &= (-1)^n[\phi_L,[\mu,f]]\\
                            &= (\delta\circ \d)(f)
\end{align*}
\end{proof}
We are now in a position to define the cohomology of the 3-LieDer pair. 
We define a map $\partial: C^n_{\text{3-LieDer}} (L,M)\rightarrow C^{n+1}_{\text{3-LieDer}} (L,M)$ by
\begin{eqnarray*}
&&\partial f=(d f,  -\delta f), ~~~\mbox{for all}~~~ f\in C^1_{\text{3-LieDer}} (L,M),\\
&&\partial (f_n,  \overline{f_n}) =(df_n,  d\overline{f}_n+(-1)^{n}\delta f_n),~~~\mbox{for all}~~~(f_n,  \overline{f}_n)\in C^{n}_{\text{3-LieDer}} (L,M).
\end{eqnarray*}
\begin{proposition}
The map $\partial$ satisfies $\partial\circ \partial=0$.
\end{proposition}
{\bf Proof.}  For any $f\in C^1_{\text{3-LieDer}} (L,M)$, we have
\begin{eqnarray*}
(\partial\circ \partial) f= \partial (d f,  -\delta f)=((d\circ d)f,  -(d\circ\delta) f+ (\delta\circ d) f )=0.
\end{eqnarray*}
Similarly, for any $(f_n,  \overline{f}_n)\in C^{n}_{\text{3-LieDer}} (L,M)$, we have
\begin{eqnarray*}
(\partial\circ \partial) (f_n,  \overline{f_n})&=&\partial (df_n, df_n+(-1)^{n}f_n)\\
&=& (d^2f_n, d^2\overline{f_n}+(-1)^{n}d\delta f_n+(-1)^{n+1}\delta df_n)\\
&=&0.
\end{eqnarray*}
Hence the proof is finished. \hfill $\square$

Therefore, $(C^{\ast}_{\text{3-LieDer}} (L,M), \partial)$ forms a cochain complex. We denote the corresponding cohomology groups by $H^{\ast}_{\text{3-LieDer}} (L,M)$.

 \section{ Central extensions of 3-LieDer pairs}
\def\theequation{\arabic{section}. \arabic{equation}}
\setcounter{equation} {0}
In this section, we study central extensions of a 3-LieDer pair. Similar to the classical cases, we show that isomorphic classes of central extensions are classified by the second cohomology of the 3-LieDer pair with coefficients in the trivial representation.

 Let $(L,  \phi_L)$ be a 3-LieDer pair and  $(M,  \phi_{M})$ be an abelian 3-LieDer pair  i.e, the 3-Lie algebra  bracket of $M$
is trivial.
\begin{definition}
A central extension of $(L,  \phi_L)$ by $(M,  \phi_{M})$  is an exact sequence of
3-LieDer pairs
\begin{eqnarray}
\xymatrix@C=0.5cm{
  0 \ar[r] & (M,  \phi_{M}) \ar[rr]^{i} && (\hat{L}, \phi_{\hat{L}}) \ar[rr]^{p} && (L,  \phi_L) \ar[r] & 0 }
\end{eqnarray}
such that $[i(m), \hat{x}, \hat{y}]=0$, for all $m\in M$ and $\hat{x}, \hat{y}\in \hat{L}$.
\end{definition}

In a central extension, using the map $i$ we can identify $M$ with the corresponding subalgebra of $\hat{L}$ and with this  $\phi_{M}=\phi_{\hat{L}}|_{M}$.
\begin{definition}
 Two central extensions $(\hat{L},  \phi_{\hat{T}})$ and $(\hat{L'},  \phi_{\hat{L'}})$ are said to be isomorphic if there is an isomorphism
 $\eta: (\hat{L},  \phi_{\hat{L}})\rightarrow (\hat{L'},  \phi_{\hat{L'}})$ of 3-LieDer pairs that makes
the following diagram commutative

\begin{eqnarray*}
\aligned
\xymatrix{
0  \ar[rr] && (M,  \phi_{M}) \ar[d]^{Id_M}\ar[rr]^{i} & & (\hat{L},  \phi_{\hat{L}})  \ar[d]^{\eta} \ar[rr]^{p} & & (L,  \phi_L) \ar[d]^{Id_L} \ar[rr] & & 0 \\
0 \ar[rr]& & (M,  \phi_{M}) \ar[rr]^{i'} && (\hat{L'}, \phi_{\hat{L'}}) \ar[rr]^{q} & & (L,  \phi_L)  \ar[rr] & & 0.}
\endaligned
\end{eqnarray*}
\end{definition}
Let Eq.(3.1) be a central extension of $(L,  \phi_L)$. A section of the map $p$ is given by a linear map $s : L\rightarrow \hat{L}$ such
that $p\circ s=Id_L$.

For any section $s$, we define linear maps $\psi: L\wedge L\wedge L \rightarrow M$ and  $\chi: L\rightarrow M$ by
\begin{eqnarray*}
\psi(x, y, z):=[s(x), s(y), s(z)]-s([x, y, z]),~~~~~\chi(x)=\phi_{\hat{L}}(s(x))-s(\phi_L(x)),~~\text{for all}~ x, y, z\in L.
\end{eqnarray*}

Note that the vector space $\hat{L}$ is isomorphic to the direct sum $L\oplus M$ via the section $s$. Therefore,
we may transfer the structures  of $\hat{L}$ to $L\oplus M$. The product and  linear
maps on $L\oplus M$ are given by
\begin{eqnarray*}
&&[(x, m), (y, n), (z, p)]_{\psi}=([x, y, z], \psi(x, y, z)), \\
&& \phi_{L\oplus M}(x, m)=(\phi_L(x), \phi_{M}(m)+\chi(x)).
\end{eqnarray*}
\begin{proposition}
 The vector space $L\oplus M$ equipped with the above product and linear maps $\phi_{L\oplus M}$ forms a 3-LieDer pair if and only if $(\psi, \chi)$ is a 2-cocycle in the
cohomology of the 3-LieDer pair $(L, \phi_L)$ with coefficients in the trivial representation $M$. Moreover, the
cohomology class of $(\psi,  \chi)$ does not depend on the choice of the section $s$.
\end{proposition}
{\bf Proof.}  The tuple $(L \oplus M,  \phi_{L\oplus M})$ is a 3-LieDer pair if and only if the following equations holds:
\begin{eqnarray}
&&[(x, m), (y, n), [(z, p), (v, k), (w, l)]_{\psi}]_\psi \nonumber\\
&=&  [[(x, m),(y,n),(z,p)]_\psi , (v,k), (w,l)]_\psi + [ (z,p),[(x,m),(y,n),(v,k)]_\psi ,(w,l)]_\psi \nonumber\\
&&+[(z,p), (v,k),[(x,m),(y,n),(w,l)]_\psi ]_\psi,\\
and,\\
&& \phi_{L\oplus M}[(x, m), (y, n), (z, p)]_{\psi}\nonumber\\
&=& [\phi_{L\oplus M}(x, m), (y, n), (z, p)]_{\psi}+[(x, m), \phi_{L\oplus M}(y, n), (z, p)]_{\psi}\nonumber\\
&&+[(x, m), (y, n), \phi_{L\oplus M}(z, p)]_{\psi},
\end{eqnarray}
for all $x \oplus m, y\oplus n, z\oplus p, v\oplus k, w\oplus l\in L\oplus M$. The condition Eq.(3.2) is equivalent to
\begin{eqnarray*}
 \psi(x, y, [z, v, w]) = \psi([x, y, z], v, w)+ \psi(z, [x, y, v], w)+\psi(z, v,  [x, y, w]),
\end{eqnarray*}
or, equivalently, $d(\psi)=0$, as we are considering only trivial representation. The condition Eq.(3.3) is equivalent to
\begin{eqnarray*}
&& \phi_{M}(\psi(x, y, z))+\chi([x, y, z])= \psi(\phi_L(x), y, z)+\psi(x, \phi_L(y), z)+\psi(x, y, \phi_L(z)).
\end{eqnarray*}
This is same as $d(\chi) + \delta\psi = 0$. This implies $(\psi, \chi)$ is a $2$-cocycle.

Let $s_1, s_2$ be two sections of $p$. Define a map $u: L \rightarrow  M$ by $u(x):= s_1(x)-s_2(x)$.  Observe that
\begin{eqnarray*}
&&\psi(x, y, z)\\
&=&[s_1(x), s_1(y), s_1(z)]-s_1([x, y, z])\\
&=& [s_2(x)+u(x), s_2(y)+u(y), s_2(z)+u(z)]-s_2([x, y, z])-u([x, y, z])\\
&=& \psi'(x, y, z)-u([x, y, z]),
\end{eqnarray*}
as $u(x), u(y), u(z)\in M$ and $(M,\phi_M)$ be an abelian $3$-LieDer pair.

Also note that
\begin{eqnarray*}
\chi(x)&=&\phi_{\hat{L}}(s_1(x))-s_1(\phi_L(x))\\
&=&\phi_{\hat{L}}(s_2(x)+u(x))-s_2(\phi_L(x))-u(\phi_L(x))\\
&=& \chi'(x)+\phi_{M}(u(x))-u(\phi_L(x)).
\end{eqnarray*}
This shows that  $(\psi,  \chi)-(\psi',  \chi')=\partial u$. Hence they correspond to the same
cohomology class.  \hfill $\square$
\begin{theorem}
 Let $(L,  \phi_L)$ be a 3-LieDer pair and  $(M, \phi_{M})$ be an abelian 3-LieDer pair. Then the isomorphism classes of central extensions of $L$ by $M$ are classified by the second cohomology group $H^{2}_{\text{3-LieDer}} (L, M)$.
\end{theorem}
{\bf Proof.}   Let $(\hat{L},  \phi_{\hat{L}})$ and $(\hat{L'}, \phi_{\hat{L'}})$  be two isomorphic central extensions and the isomorphism is given by $\eta: \hat{L}\rightarrow \hat{L'}$. Let $s:L\rightarrow \hat{L}$ be a section of $p$. Then
\begin{eqnarray*}
p'\circ (\eta\circ s)=(p'\circ \eta)\circ s=p\circ s=Id_L.
\end{eqnarray*}
This shows that $s':=\eta\circ s$ is a section of $p'$. Since $\eta$ is a morphism of 3-LieDer pairs, we have
$\eta|_M = Id_M$. Thus,
\begin{eqnarray*}
\psi'(x, y, z)&=&[s'(x), s'(y), s'(z)]-s'([x, y, z])\\
&=& \eta([s(x), s(y), s(z)]-[x, y, z])\\
&=& \psi(x, y, z),
\end{eqnarray*}
and
\begin{eqnarray*}
\chi'(x)&=&\phi_{\hat{L'}}(s'(x))-s'(\phi_L(x))\\
&=&\phi_{\hat{L'}}(\eta\circ s(x))-\eta\circ s(\phi_L(x))\\
&=& \phi_{\hat{L}}(s(x))- s(\phi_L(x))\\
&=&\chi(x).
\end{eqnarray*}
Therefore, isomorphic central extensions give rise to the same 2-cocycle, hence, correspond to the same
element in $H^{2}_{3-LieDer} (L, M)$.

Conversely, let $(\psi,  \chi )$ and $(\psi',  \chi')$ be two cohomologous 2-cocycles. Therefore,
there exists a map $v: L \rightarrow  M$ such that
\begin{eqnarray*}
(\psi,  \chi)-(\psi',  \chi')=\partial v.
\end{eqnarray*}
The 3-LieDer pair structures on $L \oplus M$  corresponding to the above 2-cocycles are isomorphic via the map $\eta: L \oplus M\rightarrow L \oplus M$ given by $\eta(x, m) = (x, m+v(x))$. This proves our theorem. \hfill $\square$

 \section{Extensions of a pair of derivations}
\def\theequation{\arabic{section}. \arabic{equation}}
\setcounter{equation} {0}
It is well-known that derivations are infinitesimals of automorphisms, and a study \cite{BS17} has been done on extensions of a pair of automorphisms of Lie-algebras. In this section, we study extensions of a pair of derivations and see how it is related to the cohomology of the 3-LieDer pair.

Let
\begin{eqnarray}
\xymatrix@C=0.5cm{
  0 \ar[r] & M \ar[rr]^{i} && \hat{L} \ar[rr]^{p} && L \ar[r] & 0 }
\end{eqnarray}
be a fixed central extensions of 3-Lie algebras. Given a pair of derivations $(\phi_L, \phi_M)\in Der(L)\times Der(M)$,
here we study extensions of them to a derivation $\phi_{\hat{L}}\in Der(\hat{L})$ which makes
\begin{eqnarray}
\xymatrix@C=0.5cm{
  0 \ar[r] & (M,  \phi_{M}) \ar[rr]^{i} && (\hat{L}, \phi_{\hat{L}}) \ar[rr]^{p} && (L,  \phi_L) \ar[r] & 0 }
\end{eqnarray}
into an exact sequence of 3-LieDer pairs.   In such a case, the pair $(\phi_L, \phi_M)\in Der(L)\times Der(M)$ is said to be extensible.

Let  $s: L\rightarrow \hat{L}$ be a section of $Eq.(4.1)$, we define a map $\psi: L\o L\o L \rightarrow M$  by
\begin{eqnarray*}
\psi(x, y, z):=[s(x), s(y), s(z)]-s([x, y, z]),~~~~~\chi(x)=\phi_{\hat{L}}(s(x))-s(\phi_L(x)),~~\forall x, y, z\in L.
\end{eqnarray*}

Given a pair of derivations $(\phi_L, \phi_M)\in Der(L)\times Der(M)$, we define another map $Ob_{(\phi_L, \phi_M)}^{\hat{L}}: L\o L\o L \rightarrow M$ by
\begin{eqnarray*}
Ob_{(\phi_L, \phi_M)}^{M}( x, y, z):=\phi_{M}(\psi(x, y, z))-\psi(\phi_L(x), y, z)-\psi(x, \phi_L(y), z)-\psi(x, y, \phi_L(z)).
\end{eqnarray*}
\begin{proposition}
The map $Ob_{(\phi_L, \phi_M)}^{\hat{L}}: L\o L\o L \rightarrow M$ is a 2-cocycle in the cohomology of the 3-Lie algebra
$L$ with coefficients in the trivial representation a. Moreover, the cohomology class $[Ob_{(\phi_L, \phi_M)}^{\hat{L}} ] \in H^2(L, M)$ does not depend on the choice of sections.
\end{proposition}
{\bf Proof.}   First observe that $\psi$ is a 1-cocycle in the cohomology of the 3-Lie  algebra $L$ with coefficients
in the trivial representation $M$. Thus, we have
\begin{eqnarray*}
&&(dOb_{(\phi_L, \phi_M)}^{M})( x, y, u, v, w)\\
&=& -Ob_{(\phi_L, \phi_M)}^{M}(x, y, [u,v,w])+Ob_{(\phi_L, \phi_M)}^{M}([x,y,u], v, w)+Ob_{(\phi_L, \phi_M)}^{M}(u, [x,y,v], w)\\
&&+Ob_{(\phi_L, \phi_M)}^{M}(u, v,[x,y,w])\\
&=&  -\phi_{M}(\psi(x, y, [u,v,w]))+\psi(\phi_L(x), y, [u,v,w])+\psi(x, \phi_L(y), [u,v,w])\\
&&+\psi(x, y, \phi_L([u,v,w]))+\phi_{M}(\psi([x,y,u], v, w))-\psi(\phi_L([x,y,u]), v, w) \\
&&-\psi([x,y,u], \phi_L(v), w)-\psi([x,y,u], v, \phi_L(w)) +\phi_{M}(\psi(u, [x,y,v], w))\\
&& -\psi(\phi_L(u), [x,y,v], w)-\psi(u, \phi_L([x,y,v]), w)-\psi(u, [x,y,v], \phi_L(w)) \\
&&+\phi_{M}(\psi(u, v,[x,y,w]))-\psi(\phi_L(u), v, [x,y,w])-\psi(u, \phi_L(v), [x,y,w])\\
&&-\psi(u, v, \phi_L([x,y,w])) \\
&=& \psi(\phi_L(x), y, [u,v,w])+\psi(x, \phi_L(y), [u,v,w])+\psi(x, y, \phi_L([u,v,w]))\\
&&-\psi(\phi_L([x,y,u]), v, w) -\psi([x,y,u], \phi_L(v), w)-\psi([x,y,u], v, \phi_L(w)) \\
&& -\psi(\phi_L(u), [x,y,v], w)-\psi(u, \phi_L([x,y,v]), w)-\psi(u, [x,y,v], \phi_L(w)) \\
&&-\psi(\phi_L(u), v, [x,y,w])-\psi(u, \phi_L(v), [x,y,w])-\psi(u, v, \phi_L([x,y,w])) \\
&=&0.
\end{eqnarray*}
Therefore, $Ob_{(\phi_L, \phi_M)}^{\hat{L}}$ is a 2-cocycle. To prove  the second part, let $s_1$ and $s_2$ be two sections of Eq.(4.1).
Consider the map $u:L\rightarrow M$ given by $u(x):= s_1 (x)-s_2 (x)$. Then
\begin{eqnarray*}
\psi_1(x, y, z)=\psi_2(x, y, z)-u[x, y, z].
\end{eqnarray*}
If $^1Ob_{(\phi_L, \phi_M)}^{\hat{L}}$ and $^2Ob_{(\phi_L, \phi_M)}^{\hat{L}}$ denote the one cocycles corresponding to the sections $s_1$ and $s_2$, then
\begin{eqnarray*}
&& ^1Ob_{(\phi_L, \phi_M)}^{M}(x, y, z)\\
&=& \phi_{M}(\psi_1(x, y, z))-\psi_1(\phi_L(x), y, z)-\psi_1(x, \phi_L(y), z)-\psi_1(x, y, \phi_L(z))\\
&=& \phi_{M}(\psi_2(x, y, z))- \phi_{M}(u(x, y, z))-\psi_2(\phi_L(x), y, z)+u(\phi_L(x), y, z)\\
&&-\psi_2(x, \phi_L(y), z)+u(x, \phi_L(y), z)-\psi_2(x, y, \phi_L(z))+u(x, y, \phi_L(z))\\
&=&  ^2Ob_{(\phi_L, \phi_M)}^{M}(x, y, z)+d(\phi_M\circ u-u\circ \phi_L)(x, y, z).
\end{eqnarray*}
This shows that the 2-cocycles $^1Ob_{(\phi_L, \phi_M)}^{\hat{L}}$ and $^2Ob_{(\phi_L, \phi_M)}^{\hat{L}}$ are cohomologous. Hence they correspond
to the same cohomology class in $\in H^2(L, M)$.

The cohomology class $[Ob_{(\phi_L, \phi_M)}^{\hat{L}} ] \in H^2(L, M)$ is called the obstruction class to extend the pair of
derivations $(\phi_L, \phi_M)$.
\begin{theorem}
 Let Eq.(4.1) be a central extension of 3-Lie algebras. A pair of derivations $(\phi_L, \phi_M)\in Der(L)\times Der(M)$ is extensible if and only if the obstruction class $[Ob_{(\phi_L, \phi_M)}^{\hat{L}} ] \in H^2(L, M)$ is trivial.
\end{theorem}
{\bf Proof.}  Suppose there exists a derivations $\phi_{\hat{L}}\in Der(\hat{L})$ such that Eq. (4.2) is an exact sequence of 3-LieDer pairs. For any $x \in L$, we observe that $p(\phi_{\hat{L}} (s(x)) - s(\phi_{L} (x))) = 0$. Hence $\phi_{\hat{L}} (s(x))- s(\phi_{L} (x))\in
ker(p) = im(i)$. We define $\lambda: L\rightarrow M$ by
\begin{eqnarray*}
\lambda(x)=\phi_{\hat{L}} (s(x))- s(\phi_{L} (x)).
\end{eqnarray*}
For any $s(x) + a \in  \hat{L}$, we have
\begin{eqnarray*}
 \phi_{\hat{L}} (s(x) + a)=s(\phi_{L} (x))+\lambda(x)+\phi_{\hat{L}} (a).
\end{eqnarray*}
Since $\phi_{\hat{L}}$ is a derivation, for any $s(x) + a, s(y) + b \in \hat{L}$, we have
\begin{eqnarray*}
\phi_M (\psi(x, y, z)) - \psi(\phi_L (x), y, z) - \psi(x,\phi_L (y), z)- \psi(x, y, \phi_L (z))= -\lambda([x, y, z]),
\end{eqnarray*}
or,  equivalently,  $Ob_{(\phi_L, \phi_M)}^{\hat{L}}=\partial \lambda$  is a coboundary. Hence the obstruction class $[Ob_{(\phi_L, \phi_M)}^{\hat{L}} ] \in H^2(L, M)$ is trivial.

To prove the converse part, suppose $Ob_{(\phi_L, \phi_M)}^{\hat{L}}$
is given by a coboundary, say $Ob_{(\phi_L, \phi_M)}^{\hat{L}}=\partial \lambda$. We
define a map $\phi_{\hat{L}}: \hat{L}\rightarrow \hat{L}$ by
\begin{eqnarray*}
 \phi_{\hat{L}} (s(x) + a)=s(\phi_{L} (x))+\lambda(x)+\phi_{\hat{L}} (a).
\end{eqnarray*}
Then $\phi_{\hat{L}}$ is a derivation on $\hat{L}$ and Eq. (4.2) is an exact sequence of 3-LieDer pairs. Hence the pair $(\phi_L, \phi_M)$
is extensible.
Thus, we obtain the following.
\begin{theorem}
 If $H^2(L, M)=0$, then any pair of derivations $(\phi_L, \phi_M)\in Der(L)\times Der(M)$ is extensible.
\end{theorem}
 \section{Formal deformations of 3-LieDer pairs}
\def\theequation{\arabic{section}. \arabic{equation}}
\setcounter{equation} {0}

In this section, we study one-parameter formal deformations of 3-LieDer pairs in which we deform
both the 3-Lie bracket and the distinguished derivations.

Let $(L,  \phi_L)$ be a 3-LieDer pair.  We denote the 3-Lie  bracket on $L$ by $\mu$, i.e, $\mu(x,y, z) = [x, y, z]$, for
all $x, y, z \in L$. Consider the space $L[[t]]$ of formal power series in $t$ with coefficients from $L$. Then $L[[t]]$
is a $\mathbb{F}[[t]]$-module.

A formal one-parameter deformation of the 3-LieDer pair $(L,  \phi_L)$ consist of formal power
series
\begin{eqnarray*}
 \mu_t&=&\sum_{i=0}^{\infty}t^i\mu_i\in \mbox{Hom}(L^{\otimes 3}, L)[[t]] ~\mbox{with}~\mu_0=\mu,\\
\phi_{t}&=&\sum_{i=0}^{\infty}t^i\phi_{i}\in \mbox{Hom}(L, L)[[t]] ~\mbox{with}~\phi_{0}=\phi_L,
\end{eqnarray*}
such that $L[[t]]$ together with the bracket $\mu_t$ forms a 3-Lie algebra over $\mathbb{F}[[t]]$ and $\phi_t$ is a derivation on $L[[t]]$.

Therefore, in a formal one-parameter deformation of 3-LieDer pair, the following relations hold:
\begin{eqnarray}
\label{5.1}&&\mu_t(x, y, \mu_t(z, v, w)) = \mu_t(\mu_t(x, y, z), v, w)+ \mu_t(z, \mu_t(x, y, v), w)+\mu_t(z, v,  \mu_t(x, y, w)),~~~~~~~~\\
\label{5.2}&& \phi_{t}(\mu_t(x, y, z))=\mu_t(\phi_{t}(x),y, z)+\mu_t(x,\phi_{t}(y), z)+\mu_t(x, y, \phi_{t}(z)).
\end{eqnarray}
Conditions Eqs.(\ref{5.1})-(\ref{5.2}) are equivalent to the following equations:
\begin{eqnarray}
&&\sum_{i+j=n}\mu_i(x, y, \mu_j(z, v, w)) \\
&=& \sum_{i+j=n}\mu_i(\mu_j(x, y, z), v, w)+ \mu_i(z, \mu_j(x, y, v), w)+\mu_i(z, v,  \mu_j(x, y, w)),~~~~~~~~\nonumber\\
and,\nonumber \\ 
&& \sum_{i+j=n}\phi_{i}(\mu_j(x, y, z))\nonumber \\
&&=\sum_{i+j=n}\mu_i(\phi_{j}(x),y, z)+\mu_i(x,\phi_{j}(y), z)+\mu_i(x, y, \phi_{j}(z)).
\end{eqnarray}
For $n = 0$ we simply get $(L, \phi_L)$ is a 3-LieDer pair. For $n = 1$, we have
\begin{eqnarray}
&& \mu_1(x, y, [z, v, w])+[x, y, \mu_1(z, v, w)]\nonumber\\
&=&\mu_1([x, y, z], v, w)+[\mu_1(x, y, z), v, w]+[z, \mu_1(x, y, v), w]\nonumber\\
\label{5.5}&&+\mu_1(z, [x, y, v], w)+[z, v, \mu_1(x, y, w)]+\mu_1(z, v, [x, y, w]),~~~~~~~\\
and,\nonumber \\
&& \phi_1([x, y, z])+\phi_{L}(\mu_1(x, y, z))\nonumber\\
&=&\mu_1(\phi_{L}(x),y, z)+[\phi_{1}(x),y, z]+\mu_1(x,\phi_{L}(y), z)+[x,\phi_{1}(y), z]\nonumber\\
\label{5.6}&&+\mu_1(x, y, \phi_{L}(z))+[x, y, \phi_{1}(z)].
\end{eqnarray}
  The condition Eq.(\ref{5.5}) is equivalent to $d(\mu_1)=0$ whereas the condition Eq.(\ref{5.6}) is
equivalent to $ d(\phi_{1}) + \delta(\mu_1)= 0$. Therefore, we have
\begin{eqnarray*}
\partial(\mu_1, \phi_{1})=0.
\end{eqnarray*}
\begin{definition}
Let $(\mu_t, \phi_t)$ be a one-parameter formal deformation of $3$-LieDer pair $(L, \phi_L)$. Suppose $(\mu_n, \phi_n)$ is the first non-zero term of $(\mu_t, \phi_t)$ after $(\mu_0, \phi_0)$, then such $(\mu_n, \phi_n)$ is called the infinitesimal of the deformation of $(L, \phi_L)$.
\end{definition}
Hence, from the above observations,  we have the following proposition.
\begin{proposition}\label{inf prop}
 Let $(\mu_t, \phi_{t})$ be a formal one-parameter deformation of a 3-LieDer pair $(L, \phi_L)$. Then the linear term $(\mu_1, \phi_{1})$ is a 1-cocycle in the cohomology of the 3-LieDer pair $L$ with coefficients in itself.
\end{proposition}

\begin{proof}
We have showed that $$\partial(\mu_1, \phi_{1})=0.$$
If $(\mu_1, \phi_{1})$ be the first non-zero term, then we are done. If $(\mu_n, \phi_{n})$ be the first non-zero term after $(\mu_0, \phi_0)$, then exactly the same way, one can show that 
$$\partial(\mu_n, \phi_{n})=0.$$
\end{proof}

Next, we define a notion of equivalence between formal deformations of 3-LieDer pairs.
\begin{definition}
Two deformations  $(\mu_t, \phi_{t})$ and $(\mu'_t,  \phi'_{t})$ of a 3-LieDer pair $(L,  \phi_L)$  are said to be equivalent if there exists a formal isomorphism $\Phi_t=\sum_{i=0}^{\infty}t^{i}\phi_i: L[[t]]\rightarrow L[[t]]$ with $\Phi_0=Id_L$ such that
\begin{eqnarray*}
\Phi_t \circ \mu_t=\mu'_t\circ (\Phi_t\o \Phi_t\o \Phi_t), ~~~~~\Phi_t \circ \phi_t=\phi'_t \circ \Phi_t.
\end{eqnarray*}
\end{definition}
 By comparing coefficients of $t^n$ from both the sides, we have
 \begin{eqnarray*}
&& \sum_{i+j=n}  \phi_i \circ \mu_j=\sum_{p+q+r+l=n}\mu'_p\circ (\phi_q\o \phi_r\o \phi_l),\\
&&\sum_{i+j=n} \phi'_{i}\circ \phi_j=\sum_{p+q=n}\phi_p\circ \phi_{q}.
 \end{eqnarray*}
Easy to see that the above identities hold for $n = 0$. For $n = 1$, we get
 \begin{eqnarray}
&&\mu_1+\phi_1\circ \mu=\mu'_1+\mu\circ (\phi_1\o Id\o Id)+\mu\circ (Id\o  Id \o\phi_1),\\
&& \phi_L \circ \Phi_1+\phi'_{1}=\phi_{1}+\phi_1 \circ \phi_{L}.
 \end{eqnarray}
 These two identities together imply that
 \begin{eqnarray*}
 (\mu_1, \phi_{1})-(\mu'_1,  \phi'_{1})=\partial \phi_1.
 \end{eqnarray*}
 Thus, we have the following.
 \begin{proposition}
 The infinitesimals corresponding to equivalent deformations of the 3-LieDer pair $(L, \phi_L)$ are cohomologous.
 \end{proposition}

 \begin{definition}
  A deformation $(\mu_t,  \phi_{t})$ of a 3-LieDer pair is said to be trivial if it is
equivalent to the undeformed deformation $(\mu'_t=\mu,  \phi'_{t}=\phi_L)$.
 \end{definition}
  \begin{definition}
   A 3-LieDer pair $(L,\phi_L)$ is called rigid, if every 1-parameter
formal deformation $\mu_t$ is equivalent to the trivial deformation.
  \end{definition}
 \begin{theorem}
Every formal deformation of the 3-LieDer pair $(L,  \phi_L)$ is rigid if the second cohomology group of the 3-LieDer pair vanishes, that is, $H^{2}_{\text{3-LieDer}} (L, L) = 0$.
 \end{theorem}
 {\bf Proof.}  Let $(\mu_t,  \phi_{t})$ be a deformation of the 3-LieDer pair $(L,  \phi_L)$. From the Proposition \ref{inf prop}, the linear term $(\mu_1,  \phi_{1})$ is a 2-cocycle. Therefore, $(\mu_1,  \phi_{1})=\partial \Phi_1$  for some $\phi_1 \in  C^1_{\text{3-LieDer}}(L, L) = \text{Hom}(L, L)$.

We set $\Phi_t = Id_L + t\Phi_1: L[[t]]\rightarrow  L[[t]]$ and define
\begin{eqnarray}
\mu'_t=\Phi_t^{-1}\circ \mu_t\circ (\Phi_t\o \Phi_t\o \Phi_t),~~~~\phi'_{t}=\Phi_t^{-1}\circ \phi_{t}\circ \Phi_t.
\end{eqnarray}
By definition,  $(\mu'_t,  \phi'_{t})$   is equivalent to $(\mu_t,  \phi_{t})$.  Moreover, it follows from Eq.(5.7)
that
\begin{eqnarray*}
&& \mu'_t=\mu+t^2\mu'_2+\c \c \c ~~~~~~\mbox{and}~~~ \phi'_{t}=\phi_L+t^{2}\phi'_{2}+\c \c \c.
\end{eqnarray*}
In other words, the linear terms are vanish. By repeating this argument, we get $(\mu_t, \phi_t)$ is equivalent to $(\mu, \phi_L)$.    \hfill $\square$

Next, we consider finite order deformations of a 3-LieDer pair $(L, \phi_L)$, and show that how obstructions of extending a deformation of order $N$ to a deformation of order $(N+1)$  depends on the third cohomology class of the $3$-LieDer pair $(L, \phi_L)$ .
\begin{definition}
 A deformation of order $N$ of a 3-LieDer pair $(L, \phi_L)$ consist of finite sums $\mu_t = \sum_{i=0}^N t^i\mu_i$ and $\phi_t =\sum_{i=0}^N t^i\phi_i$
such that $\mu_t$ defines 3-Lie bracket on $L[[t]]/(t^{N+1})$ and $\phi_t$ is a derivation on it.
\end{definition}

Therefore, we have
\begin{eqnarray*}
&&\sum_{i+j=n}\mu_i(x, y, \mu_j(z, v, w)) \\
&=& \sum_{i+j=n}\mu_i(\mu_j(x, y, z), v, w)+ \mu_i(z, \mu_j(x, y, v), w)+\mu_i(z, v,  \mu_j(x, y, w)),~~~~~~~~\\
&and,\\
&& \sum_{i+j=n}\phi_{i}(\mu_j(x, y, z))=\sum_{i+j=n}\mu_i(\phi_{j}(x),y, z)+\mu_i(x,\phi_{j}(y), z)+\mu_i(x, y, \phi_{j}(z)),
\end{eqnarray*}
for $n = 0, 1,\ldots, N$. These identities are equivalent to
\begin{eqnarray}
&& [\mu, \mu_n]=-\frac{1}{2}\sum_{i+j=n, i, j> 0}[\mu_i, \mu_j], \\
&& -[\phi_L, \mu_n]+[\mu, \phi_n]=\sum_{i+j=n, i, j> 0}[\phi_i, \mu_j].
\end{eqnarray}
\begin{definition}
A deformation ($\mu_t = \sum_{i=0}^N t^i\mu_i, \phi_t =\sum_{i=0}^N t^i\phi_i$) of order $N$ is said to be extendable
if there is an element $(\mu_{N+1}, \phi_{N+1} )\in  C^2_{\text{3-LieDer}} (L, L)$ such that $(\mu'_t=\mu_t+t^{N+1}\mu_{N+1}, \phi'_t=\phi_t+t^{N+1}\phi_{N+1})$ is a deformation of order $N + 1$.
\end{definition}

Thus, the following two equations need to be satisfied-
\begin{eqnarray}
&&\sum_{i+j=N+1}\mu_i(x, y, \mu_j(z, v, w)) \nonumber \\
\label{eqn 5.12}&=& \sum_{i+j=N+1}\mu_i(\mu_j(x, y, z), v, w)+ \mu_i(z, \mu_j(x, y, v), w)+\mu_i(z, v,  \mu_j(x, y, w)),~~~~~~~~\\
&and,  \nonumber \\
&& \sum_{i+j=N+1}\phi_{i}(\mu_j(x, y, z))  \nonumber \\
\label{eqn 5.13}&&=\sum_{i+j=N+1}\mu_i(\phi_{j}(x),y, z)+\mu_i(x,\phi_{j}(y), z)+\mu_i(x, y, \phi_{j}(z)).
\end{eqnarray}
The above two equations can be equivalently written as
\begin{eqnarray}
\label{eqn 5.14}&& d(\mu_{N+1})=-\frac{1}{2}\sum_{i+j=N+1, i, j> 0}[\mu_i, \mu_j]=Ob^3\\
\label{eqn 5.15}&& d(\phi_{N+1})+\delta(\mu_{N+1})=-\sum_{i+j=N+1, i, j> 0}[\phi_i, \mu_j]=Ob^2.
\end{eqnarray}
Using the Equation \ref{eqn 5.14} and \ref{eqn 5.15}, it is a routine but lengthy work to prove the following proposition. Thus, we choose to omit the proof.
\begin{proposition}
 The pair $(Ob^3, Ob^2 ) \in C^3_{\text{3-LieDer}} (L, L)$ is a 3-cocycle in the cohomology of the 3-LieDer pair $(L, \phi_L )$ with coefficients in itself.
\end{proposition}

\begin{definition}
 Let $(\mu_t, \phi_t)$  be a deformation of order $N$ of a 3-LieDer pair $(L, \phi_L)$. The cohomology class $[(Ob^3, Ob^2 )]\in  H^3_{\text{3-LieDer}} (L, L)$  is called the obstruction class of $(\mu_t, \phi_t)$.
\end{definition}

\begin{theorem}
A deformation $(\mu_t, \phi_t )$ of order $N$ is extendable if and only if the obstruction class
$[(Ob^3, Ob^2 )]\in  H^3_{\text{3-LieDer}} (L, L)$ is trivial.
\end{theorem}
{\bf Proof.}  Suppose that a deformation $(\mu_t, \phi_t )$ of order $N$ of the  3-LieDer pair $(L, \phi_L)$ extends
to a deformation of order $N + 1$. Then  we have
\begin{eqnarray*}
\partial(\mu_{N+1}, \phi_{N+1})=(Ob^3, Ob^2 ).
\end{eqnarray*}
Thus, the obstruction class $[(Ob^3, Ob^2 )]\in  H^3_{\text{3-LieDer}} (L, L)$ is trivial.

Conversely, if the obstruction class $[(Ob^3, Ob^2)]\in  H^3_{\text{3-LieDer}} (L, L)$ is trivial, suppose that
\begin{eqnarray*}
(Ob^3, Ob^2 )=\partial(\mu_{N+1}, \phi_{N+1}),
\end{eqnarray*}
for some $(\mu_{N+1}, \phi_{N+1})\in C^{2}_{\text{3-LieDer}}(L, L)$. Then it follows from the above observation that $(\mu'_t=\mu_t+t^{N+1}\mu_{N+1}, \phi'_t=\phi_t+t^{N+1}\phi_{N+1})$  is a deformation of order $N + 1$, which implies that $(\mu_t, \phi_t )$ is extendable.  \hfill $\square$
\begin{theorem}
 If $H^3_{\text{3-LieDer}} (L, L)$,  then every finite order deformation of $(L, \phi_L)$ is extendable.
\end{theorem}
\begin{cor}
If $H^3_{\text{3-LieDer}} (L, L)=0$,  then every $2$-cocycle in the cohomology of the 3-LieDer pair $(L, \phi_L)$ with coefficients in itself is the infinitesimal of a formal deformation of $(L, \phi_L)$.
\end{cor}
\begin{center}
 {\bf ACKNOWLEDGEMENT}
 \end{center}

The paper is supported by the NSF of China (No. 12161013) and Guizhou Provincial  Science and Technology  Foundation (No. [2020]1Y005).

\renewcommand{\refname}{REFERENCES}

\end{document}